\documentclass[12pt,reqno]{amsart}
\usepackage{amsaddr}
\usepackage[ocgcolorlinks,hyperfootnotes=false,colorlinks=true,citecolor=blue,linkcolor=blue,urlcolor=blue]{hyperref}
\author{Tirthankar Bhattacharyya and Abhay Jindal}
\address{Department of Mathematics, 	Indian Institute of Science, \\
	Bangalore 560012, India}
\email{tirtha@iisc.ac.in; abjayj@iisc.ac.in}

\usepackage{color, bm, amscd, tikz-cd}
\usepackage{csquotes}
\newcommand{\bydef}{\stackrel{\rm def}{=}}

\usepackage{tikz}

\newcommand{\cB}{{\mathcal B}}

\newcommand{\cD}{{\mathcal D}}
\newcommand{\cE}{{\mathcal E}}
\newcommand{\cF}{{\mathcal F}}

\newcommand{\cH}{{\mathcal H}}

\newcommand{\cK}{{\mathcal K}}

\newcommand{\cM}{{\mathcal M}}

\newcommand{\cO}{{\mathcal O}}

\newcommand{\la}{\langle}
\newcommand{\ra}{\rangle}
\newcommand{\bD}{{\mathbb D}}

\newcommand{\bfT}{\textit{\textbf{T}}}
\newcommand{\bfR}{\textit{\textbf{R}}}
\newcommand{\bfM}{\textit{\textbf{M}}}

\newcommand{\bfZ}{\textit{\textbf{Z}}}

\newcommand{\bfS}{\textit{\textbf{S}}}
\newcommand{\bfB}{\textit{\textbf{B}}}
\newtheorem{thm}{Theorem}[section]

\newtheorem{corollary}[thm]{Corollary}
\newtheorem{lemma}[thm]{Lemma}

\newtheorem{proposition}[thm]{Proposition}

\newtheorem{definition}[thm]{Definition}
\newtheorem{remark}[thm]{Remark}

\newtheorem*{theorem*}{Theorem}
\newtheorem*{definition*}{Definition}

\numberwithin{equation}{section}

\def\textmatrix#1&#2\\#3&#4\\{\bigl({#1 \atop #3}\ {#2 \atop #4}\bigr)}
\def\dispmatrix#1&#2\\#3&#4\\{\left({#1 \atop #3}\ {#2 \atop #4}\right)}
\numberwithin{equation}{section}

\def\textmatrix#1&#2\\#3&#4\\{\bigl({#1 \atop #3}\ {#2 \atop #4}\bigr)}
\def\dispmatrix#1&#2\\#3&#4\\{\left({#1 \atop #3}\ {#2 \atop #4}\right)}

\begin{document}
	
	\title[Characteristic Functions]{Complete Nevanlinna-Pick kernels And The Characteristic Function}
	\maketitle
	\begin{abstract}
		This note finds a new characterization of complete Nevanlinna-Pick kernels on the Euclidean unit ball.
		
		The classical theory of Sz.-Nagy and Foias about the characteristic function is extended in this note to a commuting tuple $\bfT$ of bounded operators satisfying the natural positivity condition of $1/k$-contractivity for an irreducible unitarily invariant complete Nevanlinna-Pick kernel. The characteristic function is a multiplier from $H_k \otimes \cE$ to $H_k \otimes \cF$, {\em factoring} a certain positive operator, for suitable Hilbert spaces $\cE$ and $\cF$ depending on $\bfT$. There is a converse, which roughly says that if a kernel $k$ {\em admits} a characteristic function, then it has to be an irreducible unitarily invariant complete Nevanlinna-Pick kernel. The characterization explains, among other things, why in the literature an analogue of the characteristic function for a Bergman contraction ($1/k$-contraction where $k$ is the Bergman kernel), when viewed as a multiplier between two vector valued reproducing kernel Hilbert spaces, requires a different (vector valued) reproducing kernel Hilbert space as the domain.

	\end{abstract}


		\let\thefootnote\relax
	
	
	\section{Introduction}
	This section sets up notations and tools that will be used throughout this note. Results start from Section \ref{ExistCHFN}.
	
	\subsection{Unitarily invariant kernels}
	A {\em unitarily invariant kernel} on the open Euclidean unit ball
	$$\mathbb{B}_{d} = \{\bm z = (z_1, \ldots z_d) \in \mathbb C^d : \| \bm{z}\| \bydef (|z_1|^2 + \cdots + |z_d|^2)^{1/2} < 1\}$$
	is a reproducing kernel  $k$ of the form
	\begin{equation}\label{def k}
		k(\bm{z},\bm{w})=\sum\limits_{n=0}^{\infty}a_{n} \langle \bm{z},\bm{w}\rangle^{n} \hspace{5mm}  (\bm{z}, \bm{w} \in \mathbb{B}_{d})	
	\end{equation}
	for some sequence of strictly positive coefficients $\{a_{n}\}_{n\geq 0}$ with $a_{0} =1$. The corresponding reproducing kernel Hilbert space, to be denoted by $H_{k}$, is called a {\em unitarily invariant space}.  The {\em generalized Bergman kernels}
	\begin{equation} \label{km}
		k_{m}(\bm{z} , \bm{w}) = \left( \frac{1}{1 - \langle \bm{z} , \bm w \rangle } \right)^m; \; m=1,2, \ldots .
	\end{equation}
	are  examples of unitarily invariant kernels.  For $m=1,$ $k_{m}$ is called the Drury-Arveson kernel. The Dirichlet kernel
	$$ k(z,w) = \sum\limits_{n=0}^{\infty} \frac{1}{n+1} z^{n}\overline{w}^{n} \hspace{5mm} (z,w \in \bD )$$
	is a unitarily invariant kernel and a motivating example in this note.

	Since $a_{0} =1,$ there exists a sequence of real numbers $\{b_{n}\}_{n=1}^{\infty}$ such that
	\begin{equation}\label{bn}
		\sum\limits_{n=1}^{\infty} b_{n} \la \bm{z}, \bm{w} \ra^{n} = 1- \frac{1}{\sum\limits_{n=0}^{\infty} a_{n}\la \bm{z}, \bm{w}\ra^{n}} \hspace{5mm} \text{for $\bm{z}, \bm{w} \in \mathbb{B}_{d}$}	
	\end{equation}
	such that $\|\bm{z}\| <1$ and $\|\bm{w}\|<\epsilon$ for some $\epsilon>0$. Moreover, if $k$ is {\em non-vanishing} then this equality holds for all $\bm{z},\bm{w} \in \mathbb{B}_{d}.$ For any $\epsilon > 0,$
	$${\rm span} \{ k_{\bm{w}}:\bm{w}\in\mathbb{B}_{d}, \|\bm{w}\| < \epsilon \}$$ of the set of {\em kernel functions} defined as
	$$
	k_{\bm w}(\bm z) = k(\bm{z},\bm{w}) \hspace{5mm} \text{ for } \bm z , \bm{w} \in \mathbb{B}_{d}	
	$$
	is dense in $H_{k}.$
	
	Let $\mathbb{Z}_{+}$ denote the set of all non-negative integers.
	Given $\alpha\in\mathbb{Z}^{d}_{+}$ and $\bm{z}\in \mathbb{C}^{d},$ we use the usual multi-index notations:
	$$\alpha= (\alpha_{1}, \ldots ,\alpha_{d}), \hspace{5mm} |\alpha|=\alpha_{1}+ \cdots +\alpha_{d}, \hspace{5mm} \alpha !=\alpha_{1}! \ldots \alpha_{d}!$$ and
	$$\binom{|\alpha|}{\alpha} = \frac{|\alpha|!}{\alpha_{1}!\dots\alpha_{d}!}, \hspace{10mm} \bm z^{\alpha} = z_{1}^{\alpha_{1}}  \ldots z_{d}^{\alpha_{d}}.$$
	We set
	\begin{equation}\label{a_alpha}
		a_{\alpha} = \begin{cases}  a_{|\alpha|} \binom{|\alpha|}{\alpha}, & \alpha\in\mathbb{Z}^{d}_{+} \\ 0, & \alpha\in\mathbb{Z}^{d} \backslash \mathbb{Z}^{d}_{+} \end{cases} \text{ for } \alpha\in\mathbb{Z}^{d}, \text{ and }
		b_{\alpha} = b_{|\alpha|} \binom{|\alpha|}{\alpha} \text{ for } \alpha\in\mathbb{Z}^{d}_{+} \backslash \{0\}.
	\end{equation}
	With the help of notations in $\eqref{a_alpha}$, a unitarily invariant kernel $k,$ defined in $\eqref{def k}$ can be written as $$k(\bm{z},\bm{w}) = \sum\limits_{\alpha\in\mathbb{Z}^{d}_{+}} a_{\alpha} \bm{z}^{\alpha} \overline{\bm{w}^{\alpha}} .$$ It is clear that the monomials $\{\bm z^{\alpha}\}_{\alpha\in\mathbb{Z}^{d}_{+}}$ form an orthogonal basis for a unitarily invariant space $H_{k}$ and that $$ \|\bm z^{\alpha}\|^{2}_{H_{k}} = \frac{1}{a_{\alpha}}$$	
	for all $\alpha\in\mathbb{Z}^{d}_{+}.$
	
	For a Hilbert space $\cE$, let $\cO(\mathbb{B}_{d}, \cE)$ be the class of all holomorphic $\cE$-valued functions on $\mathbb{B}_{d}$. Then the vector valued Hilbert space $H_{k}(\cE)$ is defined as
	\begin{align*}H_{k}(\cE) := \bigg\{f\in\cO(\mathbb{B}_{d}, \cE) & : f(\bm z) = \sum\limits_{\alpha\in\mathbb{Z}^{d}_{+}} c_{\alpha} \bm z^{\alpha}, c_{\alpha} \in \cE \\
		\text{ and } & \|f\|^{2}:= \sum\limits_{\alpha\in\mathbb{Z}^{d}_{+}} \|c_{\alpha}\|^{2} \|\bm z^{\alpha}\|^{2} < \infty \bigg\}.
	\end{align*}
	As a Hilbert space, $H_{k}(\cE)$ is the same as $H_{k} \otimes \cE,$ the identification being
	$$ \sum\limits_{\alpha\in\mathbb{Z}^{d}_{+}} c_{\alpha} \bm z^{\alpha} \rightarrow  \sum\limits_{\alpha\in\mathbb{Z}^{d}_{+}} (\bm{z}^{\alpha} \otimes c_{\alpha}). $$

	\subsection{$1/k-$contractions}
	We shall denote a commuting $d-$tuple of bounded operators $(T_{1}, \dots, T_{d})$ by $\bfT.$ For $\alpha\in\mathbb{Z}^{d},$ we set $$\bfT^{\alpha} = \begin{cases}  T_{1}^{\alpha_{1}} \dots T_{d}^{\alpha_{d}}, & \alpha\in\mathbb{Z}^{d}_{+} \\ I, & \alpha\in\mathbb{Z}^{d} \backslash \mathbb{Z}^{d}_{+} \end{cases} .$$

	The following definition, motivated by \eqref{bn} in the context of a unitarily invariant kernel on $\mathbb{B}_{d}$ is from \cite{CH} and is of crucial use to us.
	\begin{definition}
		Let $k$ be a unitarily invariant kernel on $\mathbb{B}_{d}$ and $\bfT$ be a commuting $d-$tuple of bounded operators such that the series $$\sum\limits_{\alpha\in\mathbb{Z}^{d}_{+} \backslash \{0\}} b_{\alpha} \bfT^{\alpha}(\bfT^\alpha)^{*}$$ converges in strong operator topology. The $d-$tuple $\bfT$ is called a $1/k$-contraction if$$ I - \sum\limits_{\alpha\in\mathbb{Z}^{d}_{+} \backslash \{0\}} b_{\alpha} \bfT^{\alpha}(\bfT^{\alpha})^{*} \geq 0.$$ In this case, we shall denote the unique positive square root of the positive operator $ I-\sum\limits_{\alpha\in\mathbb{Z}^{d}_{+} \backslash \{0\}} b_{\alpha} \bfT^{\alpha}(\bfT^{\alpha})^{*}$ by $\Delta_{\bfT}.$
	\end{definition}
	While we shall be content with the definition above, a vast generalization of this definition, applicable to other domains, appeared in \cite{AE}. A contraction $T$ and a {\em commuting contractive tuple} (or a $d$-contraction in Arveson's terminology, see page 175 of \cite{curv} for example) $\bfT = (T_{1},  \hdots,T_{d})$ are $1/k$-contractions when $k$ is the Szeg\H{o} kernel (on $\bD$) and the Drury Arveson kernel (on $\mathbb{B}_{d}$) respectively.
	
	\begin{definition}\label{pure}
		A $1/k-$contraction $\bfT = (T_{1}, \hdots,T_{d})$ is called pure if the series  $$\sum\limits_{\alpha\in\mathbb{Z}^{d}_{+}} a_{\alpha} \bfT^{\alpha} \Delta_{\bfT}^{2}(\bfT^{\alpha})^{*}$$ converges strongly to $I.$
	\end{definition}
	
	If $k$ is the Drury-Arveson kernel and $\bfT$ is a $d$-contraction, then the series above converges strongly to $I$ if and only if $\bfT$ is pure in the sense of Definition 3.1 of \cite{BES}. The compression of a pure $1/k$-contraction to a co-invariant subspace is a pure $1/k$-contraction. Theorem \ref{V_T} will imply that any pure $1/k$-contraction is, in fact, the compression to a co-invariant subspace of a special $1/k$-contraction.
	
	\subsection{Admissible kernels}
	Let $\cE$ and $\cF$ be two Hilbert spaces. For a reproducing kernel Hilbert space $H_k$ of holomorphic functions on $\mathbb{B}_{d}$, a {\em multiplier} from $H_{k} \otimes \cE$ to $H_{k} \otimes \cF$ is a $\cB(\cE, \cF)-$valued function $\varphi$ on $\mathbb{B}_{d}$ with the property that if $f$ is in $H_{k}\otimes \cE$, then $\varphi f$ is in $H_{k}\otimes \cF$. For such a $\varphi$, we shall let $M_{\varphi}$ denote the operator of multiplication by $\varphi$. An application of the Closed Graph Theorem proves $M_{\varphi}$ to be a bounded operator. The set of all such functions is denoted by $Mult(H_{k} \otimes \cE, H_{k} \otimes \cF)$. Since the multiplication operators by coordinate functions will serve as the model operators, we need to restrict the class of kernels we work with.
	
	\begin{definition}
		A unitarily invariant kernel $k$ is called admissible if the operators of multiplication by the co-ordinate functions $M_{z_i}$ are bounded operators on $H_{k}$ for $i=1,\ldots , d$ and the $d-$tuple $\bfM_{\bm{z}} = (M_{z_{1}}, \hdots,M_{z_{d}})$ is a $1/k$-contraction.
	\end{definition}
	
	The $M_{z_i}$ will sometimes be referred to as the \emph{shift operators}. The name is self-explanatory. The generalized Bergman kernels $k_{m}$ defined in \eqref{km} are examples of admissible kernels, see \cite{MV}. {\em In this paper, $k$ will always denote an admissible kernel}.
	
	\subsection{Complete Nevanlinna-Pick kernels}
	Complete Nevanlinna-Pick kernels arose out of a question of Quiggin. In \cite{Quiggin}, he asked for a characterization of reproducing kernel Hilbert spaces on which Pick's theorem was true. The answer came a year later when McCullough showed in \cite{McC} that there is a characterization if one considers Nevanlinna-Pick interpolation problem for matrix-valued functions, viz., that the inverse of the kernel has only one positive square. Starting with \cite{AMpaper} by Agler and McCarthy, the complete Nevanlinna-Pick kernels have been of constant interest over the last two decades - \cite{Richter Adv}, \cite{AHMcR}, \cite{CHS}, \cite{JKM} and \cite{McT} are representative publications. This list is by no means exhaustive and is growing. We shall see a new characterization of irreducible unitarily invariant complete Nevanlinna-Pick kernels on the Euclidean unit ball in Section \ref{CNPcharacterization}.

	\begin{definition}
		A reproducing kernel $s$ is said to have the $M_{m\times n}$ Nevanlinna-Pick property if, whenever $\bm{\lambda}_{1}, \ldots ,\bm{\lambda}_{N}$ are points in $\mathbb{B}_{d}$ and $W_{1}, \ldots ,W_{N}$ are $m$-by-$n$ matrices such that $$(I-W_{i}W_{j}^{*}) s(\bm{\lambda}_{i}, \bm{\lambda}_{j}) \geq 0,$$ then there exists a multiplier $\varphi$ in the closed unit ball of $$Mult(H_{s} \otimes \mathbb{C}^{n}, H_{s} \otimes \mathbb{C}^{m})$$ such that $\phi(\bm{\lambda}_{i})=W_{i},i=1, \ldots ,N.$ The kernel $s$ is said to have the complete Nevanllina-Pick property (CNPP) if it has the $M_{m \times n}$ Nevanlinna-Pick property for all positive integers $m$ and $n$ and the corresponding reproducing kernel Hilbert space $H_{s}$ is called a complete Nevanlinna-Pick space.
	\end{definition}
	
	\begin{definition}
		A reproducing kernel Hilbert space $H_{s}$ is irreducible if $s(\bm z, \bm w)\neq 0$ for all $\bm z,\bm w\in \mathbb{B}_{d}$, and $s_{\bm w}$ and $s_{ \bm v}$ are linearly independent if $\bm v\neq \bm w$.
	\end{definition}
	
	One of the main tools for us is the following well-known result. See Lemma 2.3 of \cite{H} for a proof. For one variable, the result is in Lemma 7.33 of \cite{AM}.
	
	\begin{lemma} \label{CNPK}
		Let $H_{s}$ be a unitarily invariant space on $\mathbb{B}_{d}$ with reproducing kernel $$ s(\bm{z},\bm{w}) = \sum\limits_{n=0}^{\infty}a_{n}\langle \bm{z}, \bm{w} \rangle^{n}.$$ Then the following are equivalent:
		\begin{enumerate}
			\item $H_{s}$ is an irreducible complete Nevanlinna-Pick space.
			\item The sequence $\{b_{n}\}_{n=1}^{\infty}$ defined by (\ref{bn}) is a sequence of non-negative real numbers.
		\end{enumerate}
	\end{lemma}
	
	\begin{definition}
		A reproducing kernel $s$ is called a unitarily invariant complete Nevanlinna-Pick (CNP) kernel if
		\begin{enumerate}
			\item it is of the form $$s(\bm{z}, \bm{w}) = \sum\limits_{n=0}^{\infty} a_{n} \la \bm{z}, \bm{w} \ra^{n} \hspace{5mm} (\bm{z} , \bm{w} \in \mathbb{B}_{d})$$ for a  sequence of strictly positive coefficients $\{a_{n}\}_{n \geq 0}$ with $a_{0} =1,$ and
			\item $H_{s}$ is an irreducible complete Nevanlinna-Pick space.
		\end{enumerate}
	\end{definition}
	
	Unitarily invariant CNP kernels are admissible. See, [\cite{CH}, Lemma 5.2]. {\em In this paper, $s$ will always denote a unitarily invariant CNP kernel}.
	
	In general, the generalized Bergman kernels $k_{m}$ defined in \eqref{km} are not always unitarily invariant CNP. Some basic examples of unitarly invariant CNP kernels are the Drury-Arveson kernel and the Dirichlet kernel.
	
	\begin{definition}
		For any $t\geq 0,$ we define a weighted Dirichlet space, to be denoted by $\cD_{t},$ in the following way
		$$ \cD_{t}:= \bigg\{f\in\cO(\bD,\mathbb{C}): f(z) = \sum\limits_{n=0}^{\infty} c_{n} z^{n} \text{ and } \|f\|^{2}_{\cD_{t}}:= \sum\limits_{n=0} ^{\infty} (n+1)^{t} |c_{n}|^{2} < \infty \bigg\}.$$
	\end{definition}
	Note that $\cD_{1}$ is the same as the Dirichlet space. The space $\cD_{t}$ is a reproducing kernel Hilbert space and the corresponding reproducing kernel is unitarily invariant CNP (see, [\cite{AM}, Corollary 7.41]).
	
	There are complete Nevanlinna-Pick kernels on $\mathbb{B}_{d}$ which are not unitarily invariant. For example, Dirichlet-type spaces of functions analytic in the unit disc whose derivatives are square area integrable with superharmonic weights have complete Nevanlinna–Pick reproducing kernels, see \cite{SS}, but in general, they are not always unitarily invariant.

	\section{Existence of the characteristic function} \label{ExistCHFN}
	The characteristic function introduced by Sz.-Nagy and Foias has a long history of generating beautiful mathematics, see \cite{NV} and many references therein for the classical theory.
	
	It was created in search of a complete unitary invariant for a contraction $T$ and is defined as
	$$ \theta_{T}(z) = (-T + z D_{T^{*}} (I_{\cH} - zT^{*})^{-1} D_{T})|_{\cD_{T}}, $$
	on $\bD$. The {\em defect operators} $D_{T} = (I_{\cH} - T^{*}T)^{1/2}$ and $D_{T^{*}} = (I_{\cH} - T T ^{*})^{1/2}$ satisfy $TD_T = D_{T^*}T$. Hence, the {\em defect space} $\cD_{T} = \overline{\rm Ran}D_{T}$ is mapped by the characteristic function into the other defect space $\cD_{T^{*}} = \overline{\rm Ran}D_{T^{*}}$.  The $\cB(\cD_{T}, \cD_{T^{*}})$ valued contractive holomorphic function $\theta_T$ is a complete unitary invariant for any completely non-unitary (c.n.u.) contraction $T$, i.e., a c.n.u. contraction $T$ is unitarily equivalent to another c.n.u. contraction $R$ if and only if the following diagram commutes for some unitary operators $\sigma_1: \cD_{T} \rightarrow \cD_R$ and $\sigma_2: \cD_{T^*} \rightarrow \cD_{R^*}$ and for all $z \in \bD$:
	
	\[
	\begin{CD}
		\mathcal{D}_{T}@>\theta_{T}(z)>>\mathcal{D}_{T^{*}}\\
		@V\sigma_1 VV@VV\sigma_2 V\\
		\mathcal{D}_{R}@>>\theta_{R}(z)>\mathcal{D}_{R^{*}}
	\end{CD} .
	\]
	
	At this point, it is important to explain the interplay between a contraction $T$ and the Szeg\H{o} kernel $\mathbb S (z,w) = \frac{1}{ 1-z \overline{w}}$ on $\bD.$ Clearly, $T$ is a contraction if and only if $\frac{1}{\mathbb S}(T, T^{*}) = I - T T^{*}$ is a positive operator. It is a natural question whether a characteristic function can always be associated with a $1/k$-contraction for an admissible kernel $k$. Although, the answer is yes in case of the Drury-Arveson kernel, we shall see that the answer is no for a general admissible $k$. Hence, this section will culminate in finding a necessary and sufficient condition for a $1/k$-contraction to have a characteristic function.
	
	\subsection{Factoring a positive operator}	
	The main result of this subsection is Proposition \ref{factor}. Let $ E_{0}$ be the {\em projection onto the one dimensional subspace of constant functions in $H_{k}$}.
	
	\begin{lemma}
		For an admissible kernel $k$, we have $\Delta_{\bfM_{\bm{z}}} = E_{0}$. Moreover,  the operator $d-$tuple $\bfM_{\bm{z}} = (M_{z_{1}}, \dots, M_{z_{d}})$ is pure in the sense of Definition \ref{pure}.
	\end{lemma}
	
	\begin{proof}
		The fact that $\Delta_{\bfM_{\bm{z}}}$ is $E_{0}$ can be seen by applying $\Delta_{\bfM_{\bm{z}}}$ on the linear combinations of kernel functions which form a dense set. The pureness of $\bfM_{\bm{z}}$ follows by observing that with respect to the orthonormal basis
		$$ e(\alpha) = \sqrt{a_{\alpha}} \bm z^\alpha \text{ for } \alpha \in \mathbb Z^d_+ ,$$
		the operator $\sum\limits_{|\alpha| \leq N} a_{\alpha} \bfM_{\bm{z}}^{\alpha} E_0 (\bfM_{\bm{z}}^{\alpha})^{*}$ is the projection onto the finite dimensional space spanned by  $ \{ e (\alpha) : \alpha\in\mathbb{Z}^{d}_{+}, |\alpha| \le N\}$.
		Thus, $\sum\limits_{|\alpha| \leq N} a_{\alpha}\bfM_{\bm{z}}^{\alpha} E_0 (\bfM_{\bm{z}}^{\alpha})^{*}$ is an increasing sequence of projections converging to $I$ in strong operator topology.
	\end{proof}
	
	Our next tool is a construction by Arazy and Englis in \cite{AE}, the roots of which go back to \cite{AEM}. For a proof of the following theorem, see [\cite{AE}, Theorem 1.3].
	
	\begin{thm} \label{V_T}
		Let $\bfT = (T_{1}, \hdots,T_{d})$ be a  pure $1/k$ contraction acting on a Hilbert space $\cH.$ Then the linear map $V_{\bfT}:\cH \to H_{k} \otimes \overline{\rm Ran} \Delta_{\bfT}$ given by
		$$h \mapsto \sum\limits_{\alpha\in\mathbb{Z}^{d}_{+}} a_{\alpha} \bm{z}^{\alpha} \otimes \Delta_{\bfT}(\bfT^{\alpha})^{*}h$$ is an isometry which satisfies $$V_{\bfT}^{*}(p(\bfM_{\bm{z}}) \otimes I_{\overline{\rm Ran} \Delta_{\bfT}}) = p(\bfT) V_{\bfT}^{*}$$ for all polynomials $p$ in $d$ complex variables.
	\end{thm}
	
	Let $$\bfM_{\bm{z}} \otimes I_{\overline{\rm Ran} \Delta_{\bfT}} \bydef (M_{z_1} \otimes I_{\overline{\rm Ran} \Delta_{\bfT}}, \ldots , M_{z_d} \otimes I_{\overline{\rm Ran} \Delta_{\bfT}}).$$ We shall suppress the suffix ${\overline{\rm Ran} \Delta_{\bfT}}$ when it is obvious. An important consequence of this theorem about an admissible kernel is that every pure $1/k$-contraction is unitarily equivalent to the compression of $\bfM_{\bm{z}} \otimes I$ to a co-invariant subspace, viz., the range of the isometry $V_{\bfT}$.
	
	\begin{definition}\label{factordef}
		Let $k$ be an admissible kernel on $\mathbb{B}_{d}.$ Let $\bfT = (T_{1},\dots,T_{d})$ be a commuting $d-$ tuple of bounded operators acting on a Hilbert space $\cH.$ A positive operator $X$ in $\cB(\cH)$ with closed range is called $(k,\bfT)$-factorable if there is a Hilbert space $\cE$ and a bounded linear transformation $\Theta : H_{k} \otimes \cE \to \cH$ such that $X = \Theta \Theta^{*}$ and $\Theta (M_{z_{i}} \otimes I_{\cE}) = T_{i} \Theta$ for all $i$.
	\end{definition}
	
	See \cite{curv} for a discussion on factorable operators when $k$ is the Drury-Arveson kernel and $\bfT = (T_{1}, \dots, T_{d})$ is a $d$-contraction.
	
	\begin{proposition} \label{factor}
		Assume the setup of Definition \ref{factordef}. Let $c_{i} = \|M_{z_{i}}\|^{2}.$  Then, $X$ is $(k, \bfT)-$factorable if and only if
		\begin{enumerate}
			\item $c_{i}X - T_{i}XT^{*}_{i} \geq 0 ,$ for all $i,$
			\item the series $P_{\bfT}(X):= \sum\limits_{\alpha\in\mathbb{Z}^{d}_{+} \backslash \{0\}} b_{\alpha}  \bfT^{\alpha} X (\bfT^{\alpha})^{*} $ converges strongly such that $X - P_{\bfT}(X) \geq 0,$ and
			\item the series  $\sum\limits_{\alpha\in\mathbb{Z}^{d}_{+}} a_{\alpha} \bfT^{\alpha} (X - P_{\bfT}(X)) (\bfT^{\alpha})^{*}$ converges strongly to $X.$
		\end{enumerate}
	\end{proposition}
	\begin{proof}
		Since $X$ is a positive operator with closed range, we have $${\rm Ran} X = {\rm Ran} X^{1/2}$$ where $X^{1/2}$ is the unique positive square root of the operator $X$. If the conditions (1), (2) and (3) are satisfied, define for each $i,$ a linear transformation $A_{i}: {\rm Ran} X \to {\rm Ran} X$ by $$A_{i} X^{1/2}h = X^{1/2}T_{i}^{*} h  \hspace{5mm} \text{for all } h \in \cH.$$ Each $A_{i}$ is a bounded operator as $$\|A_{i} X^{1/2} h \|^{2} = \langle T_{i}XT_{i}^{*} h , h \rangle \leq c_{i} \|X^{1/2} h \|^{2}.$$ Let $S_{i} = A_{i}^{*}.$ The operator $S_{i}: {\rm Ran} X \to {\rm Ran} X$ is a bounded operator which satisfies $$ T_{i} X ^{1/2} = X^{1/2} S_{i}.$$ In fact, the $d-$tuple $\bfS = (S_{1},\dots,S_{d})$ on ${\rm Ran} X$ is a $1/k$-contraction. To see that, we observe that for any $h\in\cH$ we have
		\begin{align*}
			\langle X^{1/2}h, X^{1/2}h \rangle
			& \geq  \left\langle \sum\limits_{\alpha\in\mathbb{Z}^{d}_{+} \backslash \{0\}} b_{\alpha}  \bfT^{\alpha} X (\bfT^{\alpha})^{*}   h ,h \right\rangle \hspace{5mm} \text{(by condition (2))}\\
			& = \sum\limits_{\alpha\in\mathbb{Z}^{d}_{+} \backslash \{0\}} b_{\alpha}  \langle \bfT^{\alpha} X (\bfT^{\alpha})^{*}  h, h \rangle \\
			& = \sum\limits_{\alpha\in\mathbb{Z}^{d}_{+} \backslash \{0\}} b_{\alpha} \langle  X^{1/2} \bfS^{\alpha}  (\bfS^{\alpha})^{*} X^{1/2} h, h \rangle\\
			& = \left\langle \sum\limits_{\alpha\in\mathbb{Z}^{d}_{+} \backslash \{0\}} b_{\alpha} \bfS^{\alpha} (\bfS^{*})^{\alpha} X^{1/2} h, X^{1/2} h  \right\rangle .
		\end{align*}
		More is true, viz., $\bfS$ is pure because
		\begin{align*}
			\left\langle \sum\limits_{\alpha\in\mathbb{Z}^{d}_{+}} a_{\alpha} \bfS^{\alpha} \Delta_{\bfS}^{2} (\bfS^{\alpha})^{*} X^{1/2}h , X^{1/2} h \right\rangle
			& = \sum\limits_{\alpha\in\mathbb{Z}^{d}_{+}} a_{\alpha} \langle \bfS^{\alpha} \Delta_{\bfS}^{2} (\bfS^{\alpha})^{*}X^{1/2}h, X^{1/2}h \rangle \\
			& = \left\langle \sum\limits_{\alpha\in\mathbb{Z}^{d}_{+}} a_{\alpha} \bfT^{\alpha} (X - P_{\bfT}(X)) (\bfT^{\alpha})^{*} h,h \right\rangle \\
			& =  \langle Xh, h \rangle  \hspace{5mm} \text{(by condition (3))}.
		\end{align*}
		Let $\cE = \overline{\rm Ran} \Delta_{\bfS}.$ By Theorem \ref{V_T} we have an isometry $$V_{\bfS} : {\rm Ran} X \to H_{k} \otimes \cE$$ such that $$V_{\bfS}^{*}(M_{z_{i}} \otimes I_{\cE}) = S_{i} V_{\bfS}^{*}.$$ Defining $\Theta: H_{k} \otimes \cE \to \cH$ by $\Theta = X^{1/2} V_{\bfS}^{*}$, we have $$ \Theta \Theta^{*} = X^{1/2} V_{\bfS}^{*} V_{\bfS} X^{1/2} = X$$ and $$\Theta (M_{z_{i}} \otimes I_{\cE}) = X^{1/2} V_{\bfS}^{*}(M_{z_{i}} \otimes I_{\cE}) = X^{1/2} S_{i} V_{\bfS}^{*} = T_{i} X^{1/2} V_{\bfS}^{*} = T_{i} \Theta .$$
		
		Conversely, if there is a Hilbert space $\cE$ and a bounded linear transformation $$\Theta : H_{k} \otimes \cE \to \cH$$ such that $X = \Theta \Theta^{*}$ and $\Theta (M_{z_{i}} \otimes I_{\cE}) = T_{i} \Theta$, then to prove (1), we note that
		\begin{align*}
			c_{i}X - T_{i} X T_{i}^{*}
			& = c_{i}\Theta \Theta^{*} - T_{i} \Theta \Theta^{*} T_{i}^{*} \\
			& = c_{i}\Theta \Theta^{*} - \Theta (M_{z_{i}} \otimes I_{\cE}) (M_{z_{i}}^{*} \otimes I_{\cE}) \Theta^{*} \\
			& = \Theta (c_{i} I - (M_{z_{i}} \otimes I_{\cE}) (M_{z_{i}}^{*} \otimes I_{\cE})) \Theta^{*}  \geq 0.
		\end{align*}
		To prove (2) and (3), we fix an $N\in\mathbb{N}$. Then, we note that
		\begin{align*}
			X - \sum\limits_{1 \leq |\alpha| \leq N} b_{\alpha} \bfT^{\alpha} X (\bfT^{\alpha})^{*}
			&  = \Theta \Theta^{*} - \sum\limits_{1 \leq |\alpha| \leq N} b_{\alpha}  \bfT^{\alpha} \Theta \Theta^{*} (\bfT^{\alpha})^{*}  \\
			& =  \Theta \left(I - \sum\limits_{1 \leq |\alpha| \leq N} b_{\alpha} (\bfM_{\bm{z}}^{\alpha} \otimes I_{\cE} ) ((\bfM_{\bm{z}}^{\alpha})^{*} \otimes I_{\cE})\right) \Theta^{*}
		\end{align*}
		and	
		\begin{align*}
			\sum\limits_{|\alpha| \leq N} a_{\alpha} \bfT^{\alpha} (X - P_{\bfT}(X)) (\bfT^{\alpha})^{*}
			& = \sum\limits_{|\alpha| \leq N} a_{\alpha} \bfT^{\alpha} (\Theta \Theta^{*} - P_{\bfT}(\Theta \Theta^{*})) (\bfT^{\alpha})^{*}\\
			& = \Theta \left( \sum\limits_{|\alpha| \leq N} a_{\alpha} (\bfM_{\bm{z}}^{\alpha} \otimes I_{\cE}) \Delta_{(\bfM_{\bm{z}} \otimes I_{\cE})} ((\bfM_{\bm{z}}^{\alpha})^{*} \otimes I_{\cE})\right) \Theta^{*}.
		\end{align*}
		Now, letting $N$ go to $\infty$ completes the proof.
	\end{proof}
	
	\subsection{An invariant subspace, and an associated tuple}
	The special case of an operator $X$ which is $(k, \bfT)-$factorable when the Hilbert space $\cH$ is of the form $H_{k} \otimes \cE$ for the same kernel $k$ and the operators $T_i$ are $M_{z_i} \otimes I_{\cE}$ is of particular interest because it is straightforward then that the {\em factor} $\Theta$ is a multipication operator.
	\begin{definition}
		A pure $1/k$-contraction $\bfT = (T_{1}, \dots, T_{d})$ is said to admit a characteristic function if there exist a Hilbert space $\cE$ and a $\cB(\cE, \overline{\rm Ran} \Delta_{\bfT})-$valued analytic function $\theta_{\bfT}$ on $\mathbb{B}_{d}$ such that $M_{\theta_{\bfT}}$ is a multipication operator from $H_{k} \otimes \cE$ to $H_{k} \otimes \overline{\rm Ran} \Delta_{\bfT}$ satisfying $$I - V_{\bfT}V_{\bfT}^{*} = M_{\theta_{\bfT}} M_{\theta_{\bfT}}^{*}.$$
	\end{definition}
	
	The following proposition follows from the fact that if an operator $\Theta: H_{k} \otimes \cE \to H_{k} \otimes \cF$ satsifies $\Theta (M_{z_{i}} \otimes I_{\cE}) = (M_{z_{i}} \otimes I_{\cF}) \Theta$ for all $i = 1, \dots ,d,$ then there exists a $\cB(\cE, \cF)-$valued analytic function $\theta$ on $\mathbb{B}_{d}$ such that $\Theta = M_{\theta}.$
	\begin{proposition} \label{chfnfac}
		A pure $1/k$-contraction $\bfT = (T_{1}, \dots, T_{d})$  admits a characteristic function if and only if the projection $I - V_{\bfT} V_{\bfT}^{*}$ is $(k, \bfM_{\bm{z}} \otimes I )-$ factorable.
	\end{proposition}
	
	By virtue of Theorem \ref{V_T}, we observe that the kernel of $V_{\bfT}^{*}$, i.e., the range of projection $I - V_{\bfT} V_{\bfT}^{*}$ is an invariant subspace of the tuple $\bfM_{\bm{z}} \otimes I$.
	
	\begin{definition} \label{assoctuple}
		Let $\bfT = (T_{1}, \dots,T_{d})$ be a pure $1/k-$contraction. The associated $d-$tuple of commuting operators $\bfB_{\bfT}$ is defined on ${\rm Ker} V_{\bfT}^{*}$ as $$\bfB_{\bfT} = ((M_{z_1} \otimes I)|_{{\rm Ker} V_{\bfT}^{*}},  \ldots , (M_{z_d} \otimes I)|_{{\rm Ker} V_{\bfT}^{*}}).$$
	\end{definition}

	We have reached the main theorem of this section which gives a neat equivalent condition for existence of a characteristic function. This is a crucial result because the proof of the characterization of unitarily invariant CNP kernels obtained in the next section requires 	this result.
	
	\begin{thm} \label{characterization}
		A pure $1/k$-contraction $\bfT = (T_{1}, \dots,T_{d})$ acting on a Hilbert space $\cH$  admits a characteristic function if and only if the associated tuple $\bfB_{\bfT}$ is a $1/k$-contraction.
	\end{thm}
	
	\begin{proof}
		It is trivial to note that $\bfB_{\bfT}$ is a $1/k$-contraction if and only if the series $$\sum\limits_{\alpha\in\mathbb{Z}^{d}_{+} \backslash \{0\}} b_{\alpha} (\bfM_{\bm{z}}^{\alpha} \otimes I) (I - V_{\bfT} V_{\bfT}^{*})((\bfM_{\bm{z}}^{\alpha})^{*} \otimes I)  $$ converges strongly such that $$ (I - V_{\bfT} V_{\bfT}^{*}) -  \sum\limits_{\alpha\in\mathbb{Z}^{d}_{+} \backslash \{0\}} b_{\alpha} (\bfM_{\bm{z}}^{\alpha} \otimes I) (I - V_{\bfT} V_{\bfT}^{*})((\bfM_{\bm{z}}^{\alpha})^{*} \otimes I) \geq 0.$$ We prove the easy side first. Let $\bfT$ admit a characteristic function. This means that there is a Hilbert space $\cE$ and a $\cB(\cE, \overline{\rm Ran} \Delta_{\bfT})-$valued bounded analytic function $\theta_{\bfT}$ on $\mathbb B_d$ such that $$I - V_{\bfT} V_{\bfT}^{*} = M_{\theta_{\bfT}} M_{\theta_{\bfT}}^*.$$ Thus, for any $N \geq 1,$ we get
		\begin{align*}
			& (I - V_{\bfT} V_{\bfT}^{*}) - \sum\limits_{1 \leq |\alpha| \leq N} b_{\alpha} (\bfM_{\bm{z}}^{\alpha} \otimes I) (I - V_{\bfT} V_{\bfT}^{*})((\bfM_{\bm{z}}^{\alpha})^{*} \otimes I)\\
			= & M_{\theta_{\bfT}} M_{\theta_{\bfT}}^{*} -  \sum\limits_{1 \leq |\alpha| \leq N} b_{\alpha}(\bfM_{\bm{z}}^{\alpha} \otimes I) M_{\theta_{\bfT}} M_{\theta_{\bfT}}^{*} ((\bfM_{\bm{z}}^{\alpha})^{*} \otimes I) \\
			= & M_{\theta_{\bfT}} \left(I -  \sum\limits_{1 \leq |\alpha| \leq N} b_{\alpha}(\bfM_{\bm{z}}^{\alpha} \otimes I)   ((\bfM_{\bm{z}}^{\alpha})^{*} \otimes I) \right)M_{\theta_{\bfT}}^{*}
		\end{align*}
		and that completes the proof of this direction.
		
		Conversely, let $c_{i} = \|M_{z_{i}}\|^{2}.$ Since $V_{\bfT}$ is an isometry, $I - V_{\bfT}V_{\bfT}^{*}$ is the projection onto ${\rm Ker} V_{\bfT}^{*}$ which is invariant under $M_{z_{i}} \otimes I$ for each $i.$ Let $P_{{\rm Ker} V_{\bfT}^{*}}$ be the projection of $H_{k} \otimes \overline{\rm Ran} \Delta_{\bfT}$ onto ${\rm Ker} V_{\bfT}^{*}.$ Now for each $i,$ define a linear operator $R_{i} : {\rm Ker} V_{\bfT}^{*}\to {\rm Ker} V_{\bfT}^{*} $ by $R_{i} = (M_{z_{i}} \otimes I)|_{\rm Ker V_{\bfT}^{*}}.$ Note that $\|R_{i}\|^{2} \leq c_{i}.$ So we have
		\begin{align*}
			& c_{i} (I - V_{\bfT} V_{\bfT}^{*}) - (M_{z_{i}} \otimes I) ( I - V_{\bfT} V_{\bfT}^{*}) (M_{z_{i}}^{*} \otimes I)\\
			= & c_{i} P_{{\rm Ker} V_{\bfT}^{*}} - (M_{z_{i}} \otimes I) P_{{\rm Ker} V_{\bfT}^{*}} (M_{z_{i}}^{*} \otimes I)\\
			= & c_{i} P_{{\rm Ker} V_{\bfT}^{*}} - (M_{z_{i}} \otimes I) P_{{\rm Ker} V_{\bfT}^{*}} (M_{z_{i}}^{*} \otimes I) P_{{\rm Ker} V_{\bfT}^{*}} \\
			\geq & 0 \hspace{5mm} \text{(since $c_{i} I_{{\rm Ker} V_{\bfT}^{*}} - R_{i} R_{i}^{*} \geq 0$)}.
		\end{align*}
		
		Since the operator $I - V_{\bfT} V_{\bfT}^{*}$ acts on the Hilbert space $H_{k} \otimes \overline{\rm Ran} \Delta_{\bfT}$, we can consider the map $P_{\bfM_{\bm{z}} \otimes I}$ as in Proposition \ref{factor} with $\bfT$ replaced by $\bfM_{\bm{z}} \otimes I.$ For the sake of notational brevity, we shall denote it by $P$ for the rest of this proof. Then the given condition is $$(I - V_{\bfT} V_{\bfT}^{*}) - P(I - V_{\bfT} V_{\bfT}^{*}) \geq 0.$$ Next we shall show that the series $$ \sum\limits_{\alpha\in\mathbb{Z}^{d}_{+}} a_{\alpha} (\bfM_{\bm{z}}^{\alpha} \otimes I) ((I - V_{\bfT} V_{\bfT}^{*}) - P(I - V_{\bfT} V_{\bfT}^{*})) ((\bfM_{\bm{z}}^{\alpha})^{*} \otimes I)$$ converges strongly to $I- V_{\bfT}V_{\bfT}^{*}.$ Note that if we define $$S_{N} = \sum\limits_{|\alpha| \leq N} a_{\alpha} (\bfM_{\bm{z}}^{\alpha} \otimes I) ((I - V_{\bfT} V_{\bfT}^{*}) - P(I - V_{\bfT} V_{\bfT}^{*})) ((\bfM_{\bm{z}}^{\alpha})^{*} \otimes I),$$ then $S_{N}$ is an increasing sequence of positive operators. So if we can show that $\{S_{N}\}$ is bounded above then it will converge strongly.
		
		\vspace*{3mm}
		
		\textbf{Claim}: $S_{N} \leq I- V_{\bfT}V_{\bfT}^{*}$ for every $N \in \mathbb{N}.$
		
		\textbf{Proof of the claim}: Fix an $\epsilon > 0$ so that \eqref{bn} is satisfied and consider the dense subspace $$\left\{  \sum\limits_{i=1}^{l} (k_{\bm{w}_{i}} \otimes \xi_{i}) : l \ge 1, \bm w_1, \ldots \bm w_l \in \mathbb B_d, \|\bm w_1\| < \epsilon, \ldots , \|\bm w_l\| < \epsilon, \xi_1, \ldots \xi_l \in \overline{\rm Ran} \Delta_{\bfT} \right\} $$ of $H_{k} \otimes \overline{\rm Ran} \Delta_{\bfT}.$ It is enough to verify that
		\begin{equation} \label{SNleI-VV*}
			\left\langle S_{N} \sum\limits_{i=1}^{l} (k_{\bm{w}_{i}} \otimes \xi_{i}),\sum\limits_{i=1}^{l} (k_{\bm{w}_{i}} \otimes \xi_{i}) \right\rangle \le  \left\langle (I - V_{\bfT} V_{\bfT}^{*}) \sum\limits_{i=1}^{l} (k_{\bm{w}_{i}} \otimes \xi_{i}), \sum\limits_{i=1}^{l} (k_{\bm{w}_{i}} \otimes \xi_{i}) \right\rangle .
		\end{equation}
		The left hand side after a simplification is $$ \sum\limits_{|\alpha| \leq N} a_{\alpha} \sum\limits_{i,j=1}^{l} \overline{\bm{w}_{i}}^{\alpha} \bm{w}_{j}^{\alpha} \left( 1 - \sum\limits_{\beta\in\mathbb{Z}^{d}_{+} \backslash \{0\}} b_{\beta} \overline{\bm{w}_{i}}^{\beta} \bm{w}_{j}^{\beta} \right) \langle (I - V_{\bfT} V_{\bfT}^{*}) (k_{\bm{w}_{i}} \otimes \xi_{i}) , (k_{\bm{w_{j}}} \otimes \xi_{j}) \rangle .$$ This is no bigger than $$ \sum\limits_{\alpha\in\mathbb{Z}^{d}_{+}} a_{\alpha} \sum\limits_{i,j=1}^{l} \overline{\bm{w}_{i}}^{\alpha} \bm{w}_{j}^{\alpha} \left( 1 - \sum\limits_{\beta\in\mathbb{Z}^{d}_{+} \backslash \{0\}} b_{\beta} \overline{\bm{w}_{i}}^{\beta} \bm{w}_{j}^{\beta} \right) \langle (I - V_{\bfT} V_{\bfT}^{*}) (k_{\bm{w}_{i}} \otimes \xi_{i}) , (k_{\bm{w_{j}}} \otimes \xi_{j}) \rangle$$ which is the same as $$\sum\limits_{i,j=1}^{l} \langle (I - V_{\bfT} V_{\bfT}^{*}) (k_{\bm{w}_{i}} \otimes \xi_{i}) , (k_{\bm{w_{j}}} \otimes \xi_{j}) \rangle $$ and that is the right hand side of \eqref{SNleI-VV*}. This proves the claim.
		
		So $S_{N}$ converges strongly. From the calculations above it is also clear that $S_{N}$ increases to $I - V_{\bfT} V_{\bfT}^{*}.$
		
		By Proposition \ref{factor}, the operator $I - V_{\bfT} V_{\bfT}^{*}$ is $(k, \bfM_{z} \otimes I)-$factorable. Finally, by Proposition \ref{chfnfac}, the pure $1/k-$contraction $\bfT$ admits a characterisitic function.
	\end{proof}
	
	\begin{remark}
		The proof shows that the $d-$tuple $\bfB_{\bfT}$ is a pure $1/k$-contraction whenever it is a $1/k$-contraction, i.e., whenever $\bfT$ admits a characteristic function.
	\end{remark}
	
	\section{A characterization of unitarily invariant CNP kernels} \label{CNPcharacterization}
	This section contains the major result Theorem \ref{CNP characterization}. The proof crucially uses Theorem \ref{characterization}.
	
	\begin{definition} \label{admitchfn}
		An admissible kernel $k$ is said to admit a characteristic function if every pure $1/k$-contraction admits a characteristic function.
	\end{definition}
	
	\subsection{Sufficiency of the complete Nevanlinna-Pick property}
	Note  that \eqref{bn} can also be written as $$ \sum\limits_{\alpha\in\mathbb{Z}^{d}_{+}} a_{\alpha} \bm{z}^{\alpha} \overline{\bm{w}^{\alpha}} = 1 + \left(\sum\limits_{\alpha\in\mathbb{Z}^{d}_{+} \backslash \{0\}} b_{\alpha} \bm{z}^{\alpha} \overline{\bm{w}^{\alpha}} \right) \left(\sum\limits_{\alpha\in\mathbb{Z}^{d}_{+}} a_{\alpha} \bm{z}^{\alpha} \overline{\bm{w}^{\alpha}} \right).$$ This leads us to the following relation:
	\begin{equation} \label{coeff rel}
		a_{\alpha} = \sum\limits_{\beta \in\mathbb{Z}^{d}_{+} \backslash \{0\}} b_{\beta} a_{\alpha - \beta}
	\end{equation}
	for all $\alpha \in \mathbb{Z}^{d}_{+} \backslash \{0\}.$ In the series in \eqref{coeff rel}, only finitely many summands are non-zero. Also, if the kernel is a unitarily invariant CNP kernel, then all the summands are positive.
	
	Recall from \cite{CH} that if $s$ is a unitarily invariant CNP kernel then $\bfM_{\bm{z}} = (M_{z_{1}} , \dots , M_{z_{d}})$ on $H_{s}$ is a $1/s-$contraction and so is $\bfM_{\bm{z}} \otimes I_{\cE}$ for any Hilbert space $\cE.$
	
	\begin{thm}
		If $s$ is a unitarily invariant CNP kernel, then $s$ admits a characteristic function, i.e., every pure $1/s-$contraction admits a characteristic function.
	\end{thm}
	\begin{proof}
		By Theorem \ref{characterization}, it is enough to show that the restriction of $\bfM_{\bm{z}} \otimes I_{\cE}$ to a closed invariant subspace $\cM$ of $H_{s} \otimes \cE$ is a $1/s-$contraction for any Hilbert space $\cE.$   Let $P_{\cM}$ be the projection of $H_{s} \otimes \cE$ onto the subspace $\cM.$ Let $T_{i}  = (M_{z_{i}} \otimes I_{\cE})|_{\cM}$ for all $i = 1, \dots , d.$ For any $N \geq 1,$
		\begin{align*}
			\sum\limits_{1 \leq |\alpha| \leq N} b_{\alpha} \bfT^{\alpha} (\bfT^{\alpha})^{*}
			& = \sum\limits_{1 \leq |\alpha| \leq N} b_{\alpha} P_{\cM} (\bfM_{\bm{z}}^{\alpha} \otimes I_{\cE}) P_{\cM} ((\bfM_{\bm{z}}^{\alpha})^{*} \otimes I_{\cE})|_{\cM}\\
			& \leq \sum\limits_{1 \leq |\alpha| \leq N} b_{\alpha} P_{\cM} (\bfM_{\bm{z}}^{\alpha} \otimes I_{\cE}) ((\bfM_{\bm{z}}^{\alpha})^{*} \otimes I_{\cE})|_{\cM} \hspace{5mm} \text{(since $b_{\alpha} \geq 0$)}\\
			& \leq \sum\limits_{\alpha\in\mathbb{Z}^{d}_{+} \backslash \{0\}} b_{\alpha} P_{\cM} (\bfM_{\bm{z}}^{\alpha} \otimes I_{\cE}) ((\bfM_{\bm{z}}^{\alpha})^{*} \otimes I_{\cE})|_{\cM} \hspace{5mm} \text{(since $b_{\alpha} \geq 0$)}\\
			& \leq I_{\cM} \hspace{5mm} \text{(since $\bfM_{\bm{z}} \otimes I_{\cE}$ is a $1/s-$ contraction)}.
		\end{align*}
		Since this holds for all $N \geq 1,$ the $d-$tuple $\bfT = (T_{1}, \dots, T_{d})$ is a $1/s-$contraction. This completes the proof.
	\end{proof}
	
	Interestingly, the Bergman kernel does not admit a characteristic function in the sense of Definition \ref{admitchfn} as is shown below.
	
	\begin{proposition}
		The generalized Bergman kernels $k_{m}$ defined in \eqref{km} do not admit characteristic functions for $m\geq 2.$
	\end{proposition}
	\begin{proof}
		For $\bm{z} , \bm{w} \in \mathbb{B}_{d}$, the expression for the generalized Bergman kernel $k_{m}$ is $k_m(\bm{z}, \bm{w}) = \sum\limits_{\alpha\in\mathbb{Z}^{d}_{+}} \sigma_{m}(\alpha) \bm{z}^{\alpha} \overline{\bm{w}^{\alpha}}$
		where
		\begin{equation}\label{coeff}
			\sigma_{m}(\alpha) = \frac{(m+ |\alpha| -1) !}{\alpha! (m-1)!}	\hspace{5mm} (\alpha\in\mathbb{Z}^{d}_{+}).
		\end{equation}
		We shall denote the Hilbert space corresponding to the kernel $k_{m}$ by $\mathbb{H}_{m}.$ Recall that $\{ e_{m}(\alpha) = \sqrt{\sigma_{m}(\alpha)} \bm{z}^{\alpha} : \alpha\in\mathbb{Z}^{d}_{+}\}$ is an orthonormal basis for the Hilbert space $\mathbb{H}_{m}.$
		
		For $N \geq 0,$ define the subspace $\cH_{N} = {\rm span} \{  \bm{z}^{\alpha} : |\alpha| \leq N\}$ and define an operator tuple acting on $\cH_{N}$ by $\bfT_{N} = P_{\cH_{N}} \bfM_{\bm{z}}|_{\cH_{N}},$ where $P_{\cH_{N}}$ is the projection of $\mathbb{H}_{m}$ onto $\cH_{N}.$ It is easy to check that the $d-$tuple $\bfT_{N}$ is a pure $1/k_{m}-$contraction, $\Delta_{\bfT_{N}} = E_0$ and $$\cM_{N} \bydef {\rm Ker} V_{\bfT_{N}}^{*} = \overline{\rm span}\{ \bm{z}^{\alpha} : |\alpha| \geq N+1\}.$$ Let $P_{\cM_{N}}$ be the projection from $\mathbb{H}_{m}$ onto $\cM_{N}.$ Let $\bm n$ denote the multi-index $(n,0, \dots, 0).$
		
		We shall show that the pure $1/k_{m}-$contraction $\bfT_{N}$ does not admit a characteristic function for any $N \ge 0$. To that end, we consider the associated operator $d-$tuple of Definition \ref{assoctuple}. Note that
		\begin{align*}
			& \left\la \left( I_{\cM_{N}} - \sum\limits_{\alpha\in\mathbb{Z}^{d}_{+} \backslash \{0\}} b_{\alpha} \bfM_{\bm{z}}^{\alpha} P_{\cM_{N}} (\bfM_{\bm{z}}^{\alpha})^{*} \right) e_{m}(\bm{N+2}) , e_{m}(\bm{N+2}) \right\ra \\
			= & \left\la ( I - b_{\bm{1}} M_{z_{1}} P_{\cM_{N}} M_{z_{1}}^{*} )e_{m}(\bm{N+2}) , e_{m}(\bm{N+2}) \right\ra \\
			= & 1 - \sigma_{m}(\bm {1}) \frac{\sigma_{m}(\bm{N+1})}{\sigma_{m}(\bm{N+2})} \hspace{5mm} (b_{\bm{1}} = \sigma_{m}(\bm{1}) \text{ by } \eqref{coeff rel}).
		\end{align*}
		By putting values of the coefficients from \eqref{coeff} we get
		\begin{align*}
			\sigma_{m}(\bm{1}) \frac{\sigma_{m}(\bm{N+1})}{\sigma_{m}(\bm{N+2})}
			& = \frac{m!}{(m-1)!} \frac{(m + |\bm{N+1}| -1)!}{ (\bm{N+1})! (m-1)!} \frac{(\bm{N+2}) ! (m-1)!}{(m + |\bm{N+2}| -1)!}\\
			& = m \frac{(m+N)!}{(N+1)!} \frac{(N+2)!}{(m+ N+1)!}\\
			& = m \frac{N+2}{m+N+1} > 1 \hspace{5mm} \text{(for $m \geq 2$)}.
		\end{align*}
		So the associated tuple in not a $1/k_{m}-$contraction. By Theorem \ref{characterization}, the $d-$tuple $\bfT_{N}$ does not admit a characteristic function for any $N \geq 0.$ This completes the proof. \end{proof}
	
	\subsection{Necessity of the complete Nevanlinna-Pick property}
	
	The above raises the natural question of characterizing all admissible kernels which admit a characteristic function. We shall answer this question completely. Note that the $d-$tuple $(0,\dots,0)$ of identically zero operators on any Hilbert space $\cH$ is a pure $1/k$-contraction for any admissible kernel $k$ on $\mathbb{B}_{d}.$
	
	\begin{thm}\label{CNP characterization}
		Let $k$ be an admissible kernel on $\mathbb B_d$. Then the following are equivalent.
		\begin{enumerate}
			\item The $d-$tuple $(0,\dots,0)$ of identically zero operators admits a characteristic function.
			\item The kernel $k$ is unitarily invariant CNP.
			\item Any pure $1/k$-contraction admits a characteristic function.
		\end{enumerate}
	\end{thm}
	
	\begin{proof}
		
		\leavevmode
		
		$(1) \Rightarrow (2)$:  Let $\cH$ be the space of all constant functions in $H_{k}.$ Let $\bfT = (0, \dots, 0) $ be the $d-$tuple of identically zero operators on $\cH.$ Then $\Delta_{\bfT} = I_{\cH}.$ Recall that $E_{0}$ is the projection of $H_{k}$ onto the space of constant functions. If $V_{\bfT}$ is the isometry obtained in Theorem \ref{V_T}, then $V_{\bfT} V_{\bfT}^{*} = E_{0}.$ Let $e(\alpha) = \sqrt{a_{\alpha}} \bm{z}^{\alpha}$ and $\bm{n} = (n,0, \dots, 0).$ Since $\bfT$ admits a characteristic function, the associated operator tuple in Definition \ref{assoctuple} is a $1/k-$contraction, see Theorem \ref{characterization}. This means that for all $n\geq 1$ we have
		\begin{align*}
			\left\la \left( P_{\cH^{\perp}} - \sum\limits_{\alpha\in\mathbb{Z}^{d}_{+} \backslash \{0\}} b_{\alpha} \bfM_{\bm{z}}^{\alpha} P_{\cH^{\perp}} (\bfM_{\bm{z}}^{\alpha})^{*} \right) e(\bm{n+1}) , e(\bm{n+1}) \right\ra \geq 0.
		\end{align*}
		In other words,
		\begin{align*}
			0
			& \leq \left\la \left(P_{\cH^{\perp}} - \sum\limits_{i=1}^{n} b_{\bm{i}} \bfM_{\bm{z}}^{\bm{i}} P_{\cH^{\perp}} (\bfM_{\bm{z}}^{\bm{i}})^{*} \right) e(\bm{n+1}) , e(\bm{n+1}) \right\ra \\
			& = \left\la \left( P_{\cH^{\perp}} - \sum\limits_{i=1}^{n} b_{i} M_{z_{1}}^{i} P_{\cH^{\perp}} (M_{z_{1}}^{i})^{*} \right)   e(\bm{n+1}) , e(\bm{n+1}) \right\ra \\
			& = 1 - \sum\limits_{i=1}^{n} b_{\bm{i}} \frac{a_{\bm{n+1-i}}}{a_{\bm{n+1}}} \\
			& = \frac{1}{a_{\bm{n+1}}} \sum\limits_{i=1}^{n+1} b_{\bm{i}} a_{\bm{n+1-i}} - \frac{1}{a_{\bm{n+1}}} \sum\limits_{i=1}^{n} b_{\bm{i}} a_{\bm{n+1-i}} \hspace{5mm} \text{(by \eqref{coeff rel})} \\
			& = \frac{b_{\bm{n+1}}}{a_{\bm{n+1}}} = \frac{b_{n+1}}{a_{n+1}} \hspace{5mm} \text{(by \eqref{a_alpha})}.
		\end{align*}
		This implies $b_{n+1} \geq 0$ for all $n \geq 1.$ Since $b_{1} = a_{1},$ we get $b_{n} \geq 0$ for all $n \geq 1.$ Now by Lemma \ref{CNPK}, $k$ is a unitarily invariant CNP kernel. The implication $(2) \Rightarrow (3)$ has already been proved and for $(3) \Rightarrow (1)$, there is nothing to prove.
	\end{proof}

	\section{Explicit construction of the characteristic function}
	Recall that a unitarily invariant CNP kernel $s$ is a reproducing kernel of the form
	\begin{equation}\label{sdef}
		s(\bm{z}, \bm{w}) = \sum\limits_{n=0}^{\infty} a_{n} \langle \bm{z} , \bm{w} \rangle^{n}
	\end{equation}
	having the following properties:
	\begin{enumerate}
		\item $a_{0} =1$ and $a_{n} >0$ for all $n \geq 1,$ and
		\item $H_{s}$ is an irreducible complete Nevanlinna-Pick space.
	\end{enumerate}
	By Lemma \ref{CNPK}, for a unitarily invariant CNP kernel, \eqref{bn} holds for all $\bm{z}, \bm{w}\in\mathbb{B}_{d}$ and  the sequence $\{b_{n}\}_{n=1}^{\infty}$ defined by \eqref{bn} is a sequence of non-negative real numbers.
	
	\subsection{$1/s-$contractions}
	We shall give an explicit construction of the characteristic function for any $1/s-$contraction (not necessarily pure). The next lemma plays an important role in extending the notion of characteristic function beyond pure $1/s-$contractions and also clarifies the motivation for definition of a pure $1/s-$contraction.
	
	\begin{lemma}
		If $\bfT = (T_{1}, \hdots,T_{d})$ is a $1/s$-contraction, then the series $$\sum\limits_{\alpha\in\mathbb{Z}^{d}_{+}} a_{\alpha} \bfT^{\alpha} \Delta_{\bfT}^{2} (\bfT^{\alpha})^{*}$$ converges in the strong operator topology and the limiting operator is a positive contraction.
	\end{lemma}
	
	\begin{proof} The positivity is clear if we can show the convergence. We shall show that the partial sums form an increasing sequence of positive operators bounded above by the identity operator. Hence the sequence will converge in strong operator topology with the limiting operator bounded above by the identity operator. To that end, consider   $h\in\mathcal{H}$. Then for any $N \geq 1,$
		\begin{align*}
			\sum\limits_{|\alpha| \leq N} a_{\alpha} \langle \bfT^{\alpha} \Delta_{\bfT}^{2}(\bfT^{\alpha})^{*}h,h\rangle
			& = \sum\limits_{|\alpha| \leq N} a_{\alpha} \left\la \bfT^{\alpha}\left(I - \sum\limits_{\beta\in\mathbb{Z}^{d}_{+} \backslash \{0\}} b_{\beta} \bfT^{\beta}(\bfT^{\beta})^{*}\right)(\bfT^{\alpha})^{*}h,h\right\ra\\
			& \leq \sum\limits_{|\alpha| \leq N} a_{\alpha} \| (\bfT^{\alpha})^{*} h \|^{2} - \sum\limits_{1 \leq |\alpha| \leq N} \sum\limits_{\beta\in\mathbb{Z}^{d}_{+} \backslash \{0\}} a_{\alpha - \beta} b_{\beta} \| (\bfT^{\alpha})^{*} h\|^{2} \\
			&  = \sum\limits_{|\alpha| \leq N} a_{\alpha} \| (\bfT^{\alpha})^{*} h \|^{2} -  \sum\limits_{1 \leq |\alpha| \leq N} a_{\alpha} \| (\bfT^{\alpha})^{*} h \|^{2} \hspace{5mm} \text{(by \eqref{coeff rel})}.
		\end{align*}
		The last quantity is $\| h\|^{2}$ and that completes the proof.
	\end{proof}
	
	As a direct consequence of this Lemma, we get the following Corollary.
	
	\begin{corollary} \label{V_Tagain}
		Let $\bfT = (T_{1}, \hdots,T_{d})$ be a $1/s-$contraction acting on a Hilbert space $\cH.$ Define an operator $V_{\bfT}:\cH \to H_{s} \otimes \overline{\rm Ran} \Delta_{\bfT}$ by $$h \mapsto \sum\limits_{\alpha\in\mathbb{Z}^{d}_{+}} a_{\alpha} \bm{z}^{\alpha} \otimes \Delta_{\bfT}(\bfT^{*})^{\alpha}h.$$  Then $V_{\bfT}$ is a contraction which satisfies $$V_{\bfT}^{*}(p(\bfM_{\bm{z}}) \otimes I_{\overline{\rm Ran} \Delta_{\bfT}}) = p(\bfT) V_{\bfT}^{*}$$ for all polynomials $p$ in $d$ complex variables.
	\end{corollary}
	
	Note that for a pure $1/s$-contraction $\bfT$, the operator $V_{\bfT}$ defined here is the same as the one defined in Theorem \ref{V_T}. Here it is defined for a bigger class of operator tuples, albeit for a smaller class of kernels. By virtue of Corollary \ref{V_Tagain}, we can generalize Definition \ref{admitchfn} for unitarily invariant CNP kernels.
	\begin{definition} \label{admitchfn2}
		A $1/s-$contraction $\bfT = (T_{1}, \dots, T_{d})$ is said to admit a characteristic function if there exists a Hilbert space $\cE$ and a $\cB(\cE, \overline{\rm Ran} \Delta_{\bfT})-$valued analytic function $\theta_{\bfT}$ on $\mathbb{B}_{d}$ such that $M_{\theta_{\bfT}}$ is a multipication operator from $H_{s} \otimes \cE$ to $H_{s} \otimes \overline{\rm Ran} \Delta_{\bfT}$ satisfying $$ I - V_{\bfT} V_{\bfT}^{*} = M_{\theta_{\bfT}} M_{\theta_{\bfT}}^{*} .$$
	\end{definition}
	
	\subsection{Taylor spectrum}
	Let $r\mathbb{B}_{d}$ be the open ball in $\mathbb{C}^{d}$ centered at origin of radius $r$. For $r=1,$ the following lemma is due to Hartz. The proof for any $r \ge 1$ is similar to the proof of [\cite{H}, Lemma 2.3].
	
	\begin{lemma}
		Let $s$ be a unitarily invariant CNP kernel defined in \eqref{sdef}. If the power series $\sum_{n=0}^{\infty} a_{n} t^{n}$ has radius of convergence $r \geq 1,$ then $s(\bm{z}, \bm{w}) \neq 0$ for all $\bm{z}, \bm{w}\in \sqrt{r}  \mathbb{B}_{d}.$
	\end{lemma}

	\begin{corollary}\label{roc}
		Let $s$ be a unitarily invariant CNP kernel defined in \eqref{sdef}. If the power series $\sum_{n=0}^{\infty} a_{n}t^{n}$ has radius of convergence $r \geq 1,$ then the power series $\sum_{n=1}^{\infty} b_{n}t^{n}$ has radius of convergence greater than or equal to $r.$
	\end{corollary}
	
	We denote by $\sigma(\bfT) \subset \mathbb{C}^{d}$ the Taylor spectrum of a commuting $d-$tuple of bounded operators $\bfT = (T_{1}, \dots, T_{d})$ on a Hilbert space. For $r=1,$ the following lemma can be found in [\cite{CH}, Lemma 5.3]. For any $r \geq 1,$ it can be proved using the same techniques.
	
	\begin{lemma}
		Let $s$ be a unitarily invariant CNP kernel defined in \eqref{sdef} and the power series $\sum_{n=0}^{\infty} a_{n}t^{n}$ has radius of convergence $r \geq 1.$ If $\bfT = (T_{1}, \dots, T_{d})$ is a $1/s$-contraction, then $\sigma(\bfT) \subseteq \sqrt{r}\overline{  \mathbb{B}_{d}}.$
	\end{lemma}
	
	\subsection{Functional calculus}
	Note that for fixed $\bm{w} \in\mathbb{B}_{d},$ the series $\sum_{\alpha\in\mathbb{Z}^{d}_{+}} a_{\alpha} \overline{\bm{w}^{\alpha}} \bm{z}^{\alpha} $ defines an analytic function in a ball centered at origin of radius $\sqrt{r} / \| \bm{w}\| > \sqrt{r} .$ By Corollary \ref{roc}, the series $\sum_{\alpha\in\mathbb{Z}^{d}_{+} \backslash \{0\}} b_{\alpha} \overline{\bm{w}^{\alpha}} \bm{z}^{\alpha}$ also defines an analytic function in $(\sqrt{r} / \|\bm{w}\|) \mathbb{B}_{d}.$ Now by [\cite{Vas}, Theorem III.9.9], we get that the two operator series $\sum_{\alpha\in\mathbb{Z}^{d}_{+}} a_{\alpha} \overline{\bm{w}^{\alpha}} \bfT^{\alpha}$ and $\sum_{\alpha\in\mathbb{Z}^{d}_{+} \backslash \{0\}} b_{\alpha} \overline{\bm{w}^{\alpha}} \bfT^{\alpha}$ converge in norm operator topology. By virtue of \eqref{bn}, we get $$ \left(\sum\limits_{\alpha\in\mathbb{Z}^{d}_{+}} a_{\alpha} \overline{\bm{w}^{\alpha}} \bfT^{\alpha} \right) \left( I - \sum\limits_{\alpha\in\mathbb{Z}^{d}_{+} \backslash \{0\}} b_{\alpha} \overline{\bm{w}^{\alpha}} \bfT^{\alpha} \right) =  I .$$
	For $\bm{w} \in\mathbb{B}_{d},$ we set the notation: $$s_{\bm{w}}(\bfT) = \sum\limits_{\alpha \in \mathbb{Z}^{d}_{+}} a_{\alpha} \overline{\bm{w}^{\alpha}} \bfT^{\alpha}.$$
	Since the map $A \mapsto A^{*}$ is continuous in norm operator topology, we get $$(s_{\bm{z}} (\bfT))^{*} = \sum\limits_{\alpha \in \mathbb{Z}^{d}_{+}} a_{\alpha} \bm{z}^{\alpha} (\bfT^{\alpha})^{*}$$
	and $$ \left(\sum\limits_{\alpha\in\mathbb{Z}^{d}_{+}} a_{\alpha} \bm{z}^{\alpha} (\bfT^{\alpha})^{*} \right) \left( I - \sum\limits_{\alpha\in\mathbb{Z}^{d}_{+} \backslash \{0\}} b_{\alpha} \bm{z}^{\alpha} (\bfT^{\alpha})^{*} \right) =  I $$ for all $\bm{z} \in \mathbb{B}_{d}.$
	This implies
	\begin{equation} \label{inv}
		\left( I - \sum\limits_{\alpha\in\mathbb{Z}^{d}_{+} \backslash \{0\}} b_{\alpha} \bm{z}^{\alpha} (\bfT^{\alpha})^{*} \right)^{-1} = \left(\sum\limits_{\alpha\in\mathbb{Z}^{d}_{+}} a_{\alpha} \bm{z}^{\alpha} (\bfT^{\alpha})^{*} \right) = (s_{\bm{z}} (\bfT))^{*}
	\end{equation}
	
	\subsection{Construction of the characteristic function}
	Let $\bfT = (T_{1}, \dots, T_{d})$ be a $1/s-$contraction acting on a Hilbert space $\cH.$ We denote by $$\tilde{\cH} \bydef \oplus_{\alpha\in\mathbb{Z}^{d}_{+} \backslash \{0\}} \cH,$$ the infinite direct sum of the Hilbert space $\cH.$  For each multi-index $\alpha\in\mathbb{Z}^{d}_{+} \backslash \{0\},$ consider the polynomial $\psi_{\alpha} : \mathbb{B}_{d} \to \mathbb{C}$ given by
	$$\psi_{\alpha}(\bm{z}) = (b_{\alpha})^{1/2} \bm{z}^{\alpha}$$
	and define the infinite operator tuple
	\begin{equation*}\label{def Z}
		\bfZ = (\psi_{\alpha}(z) I_{\cH})_{\alpha\in\mathbb{Z}^{d}_{+} \backslash \{0\}}.
	\end{equation*}
	We shall also denote by $\bfZ$ the operator from $\tilde{\cH}$ to $\cH$ which maps $(h_{\alpha})_{\alpha\in\mathbb{Z}^{d}_{+} \backslash \{0\}}$ to $\sum\limits_{\alpha\in\mathbb{Z}^{d}_{+} \backslash \{0\}} (b_{\alpha})^{1/2} \bm{z}^{\alpha} h_{\alpha}.$ The operator $\bfZ$ is a strict contraction because
	$$\|\bfZ\|^{2} = \sum\limits_{\alpha\in\mathbb{Z}^{d}_{+} \backslash \{0\}} b_{\alpha} | \bm{z}^{\alpha}|^{2} = 1 - \frac{1}{s(\bm{z}, \bm{z})} < 1.$$
	Define by $\tilde{\bfT}$ the infinite operator tuple
	\begin{equation*}\label{def Ttilde}
		\tilde{\bfT}  =  (\psi_{\alpha}(\bfT))_{\alpha\in\mathbb{Z}^{d}_{+} \backslash \{0\}}.
	\end{equation*}
	as well as the operator from $\tilde{\cH}$ to $\cH$ which maps $(h_{\alpha})_{\alpha\in\mathbb{Z}^{d}_{+} \backslash \{0\}}$ to $\sum\limits_{\alpha\in\mathbb{Z}^{d}_{+} \backslash \{0\}} (b_{\alpha})^{1/2} \bfT^{\alpha} h_{\alpha}.$ It is easy to check that $\tilde{\bfT}$ is a contraction if and only if $\bfT$ is a $1/s$-contraction.  Also, $$\tilde{\bfT} \tilde{\bfT}^{*} = \sum\limits_{\alpha\in\mathbb{Z}^{d}_{+} \backslash \{0\}} b_{\alpha} \bfT^{\alpha} (\bfT^{\alpha})^{*} =  I_{\cH} - \Delta_{\bfT}^{2}$$
	and hence $\Delta_{\bfT}^{2} = I_{\cH} - \tilde{\bfT} \tilde{\bfT}^{*}.$ Let $D_{\tilde{\bfT}}$ be the unique positive square root of the positive operator $I_{\tilde{H}} - \tilde{\bfT}^{*} \tilde{\bfT},$ and let $\cD_{\tilde{\bfT}} = \overline{\rm Ran} D_{\tilde{\bfT}}.$ By equation (I.3.4) of \cite{NF} we obtain the identity
	\begin{equation}\label{defect}
		\tilde{\bfT}D_{\tilde{\bfT}}= \Delta_{\bfT} \tilde{\bfT}.
	\end{equation} Note that the operator $\bfZ \tilde{\bfT}^{*}$ is a strict contraction. So, $I_{\cH}-\bfZ \tilde{\bfT}^{*}$ is invertible.
	
	\begin{definition}
		The characteristic function of a $1/s$-contraction $\bfT = (T_{1}, \dots, T_{d})$ is the analytic operator valued function $\theta_{\bfT}:\mathbb{B}_{d}\to\cB(\cD_{\tilde{\bfT}}, \overline{\rm Ran} \Delta_{\bfT})$  given by
		\begin{align}
			\theta_{\bfT}(\bm{z})=(-\tilde{\bfT}+ \Delta_{\bfT}(I_{\cH}-\bfZ \tilde{\bfT}^{*})^{-1}\bfZ D_{\tilde{\bfT}})|_{\cD_{\tilde{\bfT}}}. \label{chfn}
		\end{align}
	\end{definition}
	
	The characteristic function $\theta_{\bfT}$ takes values in $\cB(\cD_{\tilde{\bfT}}, \overline{\rm Ran} \Delta_{\bfT})$ by virtue of \eqref{defect}. Since
	$$ I - \bfZ \tilde{\bfT}^{*} = I - \sum\limits_{\alpha\in\mathbb{Z}^{d}_{+} \backslash \{0\}} b_{\alpha} \bm{z}^{\alpha} (\bfT^{\alpha})^{*},$$
	we get by \eqref{inv} that
	\begin{equation*} \label{inv2}
		( I - \bfZ \tilde{\bfT}^{*})^{-1} = (s_{\bm{z}} (\bfT))^{*}.
	\end{equation*}
	
	\begin{lemma} \label{I1}
		The identity $$I-\theta_{\bfT}(\bm{z})\theta_{\bfT}(\bm{w})^{*}= \frac{1}{s(\bm{z},\bm{w})} \Delta_{\bfT} (s_{\bm{z}} (\bfT))^{*} s_{\bm{w}} (\bfT) \Delta_{\bfT}$$ holds for any $\bm{z},\bm{w}\in\mathbb{B}_{d}.$
	\end{lemma}
	
	We omit the proof because it is a straightforward computation. In what follows, $V_{\bfT}$ is as in Theorem \ref{V_T}.
	
	\begin{lemma}\label{V_T*}
		The identity $$V_{\bfT}^{*}(s_{\bm{w}}\otimes\xi)= s_{\bm{w}} (\bfT) \Delta_{\bfT} \xi$$ holds for any $\bm{w}\in\mathbb{B}_{d}$ and $\xi \in \overline{\rm Ran} \Delta_{\bfT}.$
	\end{lemma}
	\begin{proof}
		Let $h$ be any element of $\cH$. Then
		\begin{align*}
			\la V_{\bfT}^{*}(s_{\bm{w}}\otimes\xi),h\ra &= \la s_{\bm{w}}\otimes \xi, V_{\bfT}h\ra \\
			& = \la s_{\bm{w}} \otimes \xi , \sum\limits_{\alpha\in\mathbb{Z}^{d}_{+}} a_{\alpha} \bm{z}^{\alpha} \otimes \Delta_{\bfT}(\bfT^{\alpha})^{*}h \ra\\
			& = \sum\limits_{\alpha\in\mathbb{Z}^{d}_{+}} a_{\alpha} \overline{\bm{w}^{\alpha}}\la \xi , \Delta_{\bfT}(\bfT^{\alpha})^{*} h \ra \\
			& = \sum\limits_{\alpha\in\mathbb{Z}^{d}_{+}} a_{\alpha} \overline{\bm{w}^{\alpha}}\la \bfT^{\alpha} \Delta_{\bfT}\xi,h\ra \\
			& = \la \sum\limits_{\alpha\in\mathbb{Z}^{d}_{+}} a_{\alpha} \overline{\bm{w}^{\alpha}} \bfT^{\alpha}\Delta_{T}\xi,h\ra = \la s_{\bm{w}} (\bfT)\Delta_{\bfT}\xi , h \ra.
		\end{align*}
		Since $h$ is arbitrary, we have the desired identity.
	\end{proof}
	
	Using Lemma \ref{I1} and Lemma \ref{V_T*}, it is straightforward to see the following.
	\begin{corollary} \label{cor1}
		The identity $$ \left\la V_{\bfT}^{*} (s_{\bm{w}} \otimes \xi), V_{\bfT}^{*}(s_{\bm{z}} \otimes \eta) \right\ra = s(\bm{z},\bm{w}) \left\la (I- \theta_{\bfT}(\bm{z}) \theta_{\bfT}(\bm{w})^{*})\xi, \eta \right\ra $$ holds for any $\bm{z},\bm{w}\in\mathbb{B}_{d}$ and $\xi,\eta \in \overline{\rm Ran} \Delta_{\bfT}.$
	\end{corollary}

	The main result of this section is the next theorem which shows that for a $1/s$-contraction $\bfT$, the characteristic function $\theta_{\bfT}$ works as the multiplier required in Definition \ref{admitchfn2}.
	
	\begin{thm}
		Given a $1/s$-contraction $\bfT = (T_{1}, \dots,T_{d})$, its characteristic function $\theta_{\bfT}$ (defined in \eqref{chfn}) is a multiplier from $H_{s} \otimes \cD_{\tilde{\bfT}}$ to $H_{s} \otimes \overline{\rm Ran} \Delta_{\bfT}$ with $\|M_{\theta_{\bfT}}\| \leq 1$. Moreover the identity
		\begin{equation}  \label{I2}
			V_{\bfT} V_{\bfT}^{*} + M_{\theta_{\bfT}}M_{\theta_{\bfT}}^{*}  = I
		\end{equation}
		holds.
	\end{thm}
	\begin{proof}
		Define a linear map $$A : {\rm span} \{s_{\bm{w}} \otimes \xi : \bm{w}\in\mathbb{B}_{d}, \xi\in \overline{\rm Ran} \Delta_{\bfT}\} \to H_{s} \otimes \cD_{\tilde{\bfT}} $$ by
		$$A (s_{\bm{w}}\otimes\xi)=s_{\bm{w}}\otimes\theta_{\bfT}(\bm{w})^{*}\xi. $$
		For $\bm{z},\bm{w}\in\mathbb{B}_{d}$ and $\xi , \eta \in \cE$ we get
		\begin{align} \label{AisMtheta*}
			\la A (s_{\bm{w}}\otimes\xi) , A (s_{\bm{z}}\otimes\eta) \ra
			& = \la s_{\bm{w}}\otimes\theta_{\bfT}(\bm{w})^{*}\xi, s_{\bm{z}}\otimes\theta_{\bfT}(\bm{z})^{*}\eta \ra \nonumber \\
			& = s(\bm{z},\bm{w}) \la \theta_{\bfT}(\bm{z}) \theta_{\bfT}(\bm{w})^{*} \xi , \eta \ra \nonumber \\
			& = \la ( I - V_{\bfT} V_{\bfT}^{*}) (s_{\bm{w}} \otimes \xi), s_{\bm{z}} \otimes \eta \ra
		\end{align}
		by Corollary \ref{cor1}. This implies that
		$$ \| A x\| \leq \|x\|$$ for all $x \in {\rm span} \{s_{\bm{w}} \otimes \xi : \bm{w}\in\mathbb{B}_{d}, \xi\in \overline{\rm Ran} \Delta_{\bfT}\}.$ Thus, $A$ extends to be a bounded linear operator from $H_{s} \otimes \overline{\rm Ran} \Delta_{\bfT}$ to  $H_{s} \otimes \cD_{\tilde{\bfT}}.$ Now we shall prove that $A^{*} = M_{\theta_{\bfT}}.$ For $f \in H_{s} \otimes \cD_{\tilde{\bfT}}, \xi \in \overline{\rm Ran} \Delta_{\bfT}$ and $\bm{z} \in \mathbb{B}_{d},$ we have
		$$ \la (A^{*} f) (\bm{z}) , \xi \ra  = \la A^{*}f, s_{\bm{z}} \otimes \xi \ra  = \la f , A(s_{\bm{z}} \otimes \xi ) \ra  = \la f , s_{\bm{z}} \otimes \theta_{\bfT}(\bm{z})^{*} \xi \ra  = \la \theta_{\bfT}(\bm{z}) f(\bm{z}) ,  \xi \ra.$$
		So we get $(A^{*}f)(\bm{z}) = \theta_{\bfT}(\bm{z}) f(\bm{z}).$ This implies that $A^{*} = M_{\theta_{\bfT}}.$ The identity
		$$ M_{\theta_{\bfT}} M_{\theta_{\bfT}}^{*} + V_{\bfT} V_{\bfT}^{*} = I_{H_{k} \otimes \cE} $$
		follows from \eqref{AisMtheta*}.
	\end{proof}
	
	\begin{corollary}
		Given a $1/s$-contraction $\bfT$, its characteristic function $\theta_{\bfT}$ is a bounded analytic function on $\mathbb{B}_{d}$ with $\sup\limits_{\bm{z}\in\mathbb{B}_{d}}\| \theta_{\bfT}(z)\| \leq 1$.
	\end{corollary}

	\section*{Epilogue}
	In this note, we have proved the existence of characteristic function for a certain class, but did not show any utility. This will be the theme of a future paper. However, the principal utility of the characteristic function in the $pure$ case is a straightforward consequence of what we have developed so far.
	
	\begin{definition*}
		Given two $1/s$-contractions $\bfT$ and $\bfR$ on Hilbert spaces $\cH$ and $\cK$ respectively, the characteristic functions of $\bfT$ and $\bfR$ are said to coincide if there exist unitary operators $\tau:\cD_{\tilde{\bfT}}\to \cD_{\tilde{\bfR}}$ and $\tau_{*}: \overline{\rm Ran} \Delta_{\bfT} \to \overline{\rm Ran} \Delta_{\bfR}$ such that the following diagram commutes for all $\bm{z}\in\mathbb{B}_{d}$:
		\[
		\begin{CD}
			\cD_{\tilde{\bfT}}@>\theta_{\bfT}(\bm{z})>> \overline{\rm Ran} \Delta_{\bfT}\\
			@V\tau VV@VV\tau_{*}V\\
			\cD_{\tilde{\bfR}}@>>\theta_{\bfR}(\bm{z})> \overline{\rm Ran} \Delta_{\bfR}
		\end{CD}
		\]
	\end{definition*}
	
	The characteristic functions of two unitarily equivalent $1/s$-contractions clearly coincide. It is somewhat of a surprise that the converse is true for at least pure  $1/s$-contractions. This is achieved through the construction of a functional model. Using the isometry $ V_{\bfT}$ as well as the identity \eqref{I2}, we get that every pure $1/s$-contraction $\bfT = (T_{1},\hdots,T_{d})$ acting on a Hilbert space $\cH$ is unitarily equivalent to the commuting tuple $\mathbb{T} = ( \mathbb{T}_{1}, \hdots,\mathbb{T}_{d})$ on the functional space
	$$\mathbb{H}_{\bfT}=(H_{s} \otimes \overline{\rm Ran} \Delta_{\bfT}) \ominus M_{\theta_{\bfT}}(H_{s} \otimes \cD_{\tilde{\bfT}})$$
	defined by $\mathbb{T}_{i}=P_{\mathbb{H}_{\bfT}}(M_{z_{i}} \otimes I_{\cD_{\tilde{\bfT}^{*}}})|_{\mathbb{H}_{\bfT}}$ for $1 \leq i \leq d.$  This functional model produces the following theorem whose proof is along the lines of the contents of \cite{BES}.
	
	\begin{theorem*}
		Two pure $1/s$-contractions are unitarily equivalent if and only if their characteristic functions coincide.
	\end{theorem*}

\vspace*{2mm}
	
	{\small{ \begin{centerline}
				{\textsc{Acknowledgement}}
			\end{centerline}
			\noindent The seed for this work was sown by a question of Dmitry Yakubovich whose hospitality and the support by the Severo Ochoa program are gratefully acknowledged by the first named author. The first named author is supported by a J C Bose Fellowship JCB/2021/000041 of SERB. The second named author is supported by the Prime Minister's Research Fellowship PM/MHRD-20-15227.03.

\end{document}